\documentclass[oneside, 11pt]{amsart}
\usepackage{amsmath}
\usepackage{amssymb}
\usepackage{color}
\usepackage[usenames,dvipsnames,x11names,svgnames]{xcolor}
\usepackage[colorlinks=true,linkcolor=NavyBlue,citecolor=DarkGreen]{hyperref}
\usepackage{amsthm}
\usepackage{amscd}
\usepackage{geometry}
\usepackage{mathrsfs}
\usepackage{dsfont}

\newtheorem{thm}{Theorem}

\newtheorem{conj}{Conjecture}

\newtheorem{prop}{Proposition}
\newtheorem{lem}[prop]{Lemma}

\theoremstyle{definition}

\newcommand{\leqs}{\leqslant }
\newcommand{\geqs}{\geqslant }

\newcommand{\be}{\begin{equation*}}
\newcommand{\ee}{\end{equation*} }

\newcommand{\ben}{\begin{equation}}
\newcommand{\een}{\end{equation} }

\newcommand{\bs}{\begin{split}}
\newcommand{\es}{\end{split}}

\newcommand{\bmu}{\begin{multline*}}
\newcommand{\emu}{\end{multline*}}

\newcommand{\bmun}{\begin{multline}}
\newcommand{\emun}{\end{multline}}

\allowdisplaybreaks

\begin{document}
\title[The maximum of the zeta function on the 1-line]{A note on the maximum of the Riemann zeta function on the 1-line}
\author{Winston Heap}
\address{Department of Mathematics, University College London, 25 Gordon Street, London WC1H.}
\email{winstonheap@gmail.com}
\thanks{Research supported by European Research Council grant no. 670239.}

\maketitle
\begin{abstract}
We investigate the relationship between the maximum of the zeta function on the 1-line and the maximal order of $S(t)$, the error term in the number of zeros up to height $t$. We show that the conjectured upper bounds on $S(t)$ along with the Riemann hypothesis imply a conjecture of Littlewood that $\max_{t\in [1,T]}|\zeta(1+it)|\sim e^\gamma\log\log T$. The relationship in the region $1/2<\sigma<1$ is also investigated.
\end{abstract}

\section{Introduction}
The behaviour of large values of the Riemann zeta function on the 1-line was first investigated by Littlewood \cite{L}. Over the years, his lower bound has been improved several times; the current best \cite{AMM} establishes arbitrarily large values of $t$ for which \[|\zeta(1+it)|\geqs e^\gamma (\log\log t+\log\log\log t)+O(1).\] 
 
In the other direction, assuming the Riemann hypothesis he proved that for large $t$
\ben
\label{euler prod}
\zeta(1+it)\sim\prod_{p\leqs \log^2 t}\bigg(1-\frac{1}{p^{1+it}}\bigg)^{-1}
\een
from which it follows by Merten's Theorem that 
\ben
\label{littlewood}
|\zeta(1+it)|\leqs 2e^\gamma  (1+o(1))\log\log t.
\een
It is believed that the length of the Euler product can be reduced to $\log t$ and as a consequence one gets the following conjecture\footnote{For a more precise version, see the paper of Granville and Soundararajan \cite{GS}.}. 

\begin{conj}
\label{max of zeta}
We have 
\[
\max_{t\in [1,T]}|\zeta(1+it)|\sim e^\gamma\log\log T.
\]
\end{conj}

Littlewood \cite{L2} later refined the upper bound \eqref{littlewood} by replacing the constant $2e^\gamma$ by $2\beta(1)e^\gamma$ where $\beta(1)=\lim_{\sigma\to1^-}\beta(\sigma)$ and for $1/2<\sigma<1$, $\beta(\sigma)=v(\sigma)/(2-2\sigma)$ where $v(\sigma)$ is defined as the minimum exponent for which $\log \zeta(s)\ll (\log t)^{v(\sigma)}$. We will prove a stronger relation where the maximum on the 1-line is related to the behaviour on the 1/2-line.

Another object of interest in the theory of the Riemann zeta function is the remainder $S(t)$ in the formula for the number of zeros of height $t$ in the critical strip: 
\[
N(t)=\frac{t}{2\pi}\log \Big(\frac{t}{2\pi e}\Big)+\frac{7}{8}+S(t)+O(1/t).
\] 
Here, the classical bound is $S(t)\ll \log t$. Under the assumption of the Riemann hypothesis, Selberg showed that $S(t)\ll\log t/\log\log t$ which remains the current best. In terms of lower bounds, the most recent improvements are due to Bondarenko and Seip \cite{BS} who showed that conditionally there exist arbitrarily large values of $t$ for which $|S(t)|\gg\sqrt{\log t \log\log\log t/\log\log t}$.  It is generally believed that the lower bound is closer to the true maximal order of growth. 
Accordingly we define 
\[
\alpha=\limsup_{t\to\infty}\frac{\log |S(t)|}{\log\log t}
\]
and note that, conditionally, $1/2\leqs \alpha\leqs 1$. 
\begin{conj}
\label{max of s(t)}
We have \[\alpha=1/2.\]
\end{conj}
We remark that Farmer, Gonek and Hughes \cite{FGH} have made the more precise conjecture $\limsup_{t\to\infty} S(t)/\sqrt{\log t\log\log t}=1/\pi\sqrt{2}$.  
The maximum of $S(t)$ is closely related to the maximum of the zeta function. In this note we attempt to clarify this relation on the 1-line. 
\begin{thm}
\label{main thm}
Assume the Riemann hypothesis and let 
\ben
\label{L}
L(t)=\max_{|u-t|\leqs C\log\log t}|S(u)|
\een 
for $C>1/\pi$. Then for $X$ satisfying $\max(L(t)^2, \log t)=o(X)$ and large $t$ we have
\[
\zeta(1+it)=(1+o(1))\prod_{p\leqs X}\bigg(1-\frac{1}{p^{1+it}}\bigg)^{-1}.
\]
In particular, 
\[
|\zeta(1+it)|\leqs e^\gamma  (1+o(1))\log(L(t)^2+\log t).
\]
and hence Conjecture \ref{max of s(t)} together with the Riemann hypothesis implies Conjecture \ref{max of zeta}.
\end{thm}

%\begin{rem} Our proof generally follows the arguments of Littlewood and, in a similar vein, shows that the zeta function on the 1-line is well approximated by an Euler product of length approximately $\max(L(t)^2,\log t)$. 
%\end{rem}

A common approach to proving conditional upper bounds such as \eqref{littlewood} is via the explicit formula. On assuming RH one can trivially estimate the sum over zeros which leaves only a sum over primes. We aim to be more precise in this step and simply write the sum over zeros as a Stieltjes integral which allows us to exploit some cancellation from oscillating terms. This is essentially the content of the following proposition where the sum over zeros has been replaced by an integral involving $S(t)$.

\begin{prop}
\label{main prop}
Assume the Riemann hypothesis. Then, uniformly for $1/2+\delta\leqs \sigma \leqs 9/8$ with fixed $\delta>0$, and $1\leqs X\leqs e^{\sqrt{t}}$, 
we have
\ben
\label{explicit form}
-\frac{\zeta^\prime(s)}{\zeta(s)}=\sum_{n\geqs 1}\frac{\Lambda(n)}{n^s}e^{-n/X}+J(t,X)+O(X^{-1}\log t)+O(X^{1/2-\sigma})
\een
where 
\[
J(t,X)=iX^{1/2-\sigma}\int_{-t/2}^{t/2} X^{iy}\Gamma(\tfrac{1}{2}-\sigma+iy)\Big(\log X+\frac{\Gamma^\prime}{\Gamma}(\tfrac{1}{2}-\sigma+iy) \Big)S(t+y)dy.
\]
\end{prop}

We will deduce Theorem \ref{main thm} from this proposition in the next section. One can also use this formula to get upper bounds when $1/2< \sigma<1$. In this region we have the conditional upper bound 
\ben
\label{zeta RH bound}
\zeta(\sigma+it)\ll \exp\bigg(A\frac{(\log t)^{2-2\sigma}}{\log\log t}\bigg),
\een
see \cite{T}, and the unconditional lower bound
\[
\zeta(\sigma+it)=\Omega\bigg( \exp\bigg(B\frac{(\log t)^{1-\sigma}}{(\log\log t)^\sigma}\bigg)\bigg),
\]
originally due to Montgomery \cite{M}. The value of the constant $B$ has been improved several times with the current best due to Bondarenko and Seip \cite{BS0}. Again, it is generally believed that the lower bound is the true order of the maximum; indeed, based on some heuristic arguments Montgomery \cite{M} conjectured that this was the case (see \cite{Lam} for a detailed discussion). In terms of relating this to the maximum of $S(t)$ we have the following. 

\begin{thm}\label{second thm}Assume the Riemann hypothesis and let $L(t)$ be given by \eqref{L}.
Then for fixed $1/2<\sigma<1$, 
\be
\label{zeta RH bound 2}
\zeta(\sigma+it)\ll \exp\big(c\max\big(L(t)^{2-2\sigma},(\log t)^{1-\sigma}\big)(\log\log t)^{1-2\sigma}\big). 
\ee
\end{thm}

Note that the conjecture of Farmer, Gonek and Hughes gives the upper bound $\ll \exp (c(\log t)^{1-\sigma}(\log\log t)^{2-3\sigma})$ which is still a power of a double logarithm away from Montgomery's conjecture. It is possible that trivially bounding $J(t,X)$, as we do below, is too wasteful and that further cancellations are possible. A finer analysis of $J(t,X)$ would also be of interest in determining the lower order terms in Theorem \ref{main thm}.

 % Note that Conjecture \ref{max of s(t)} implies $\zeta(\sigma+it)\ll \exp\big(c(\log t)^{1-\sigma}(\log\log t)^{2-3\sigma}\big)$. 

\section{Proof of Theorems \ref{main thm} and \ref{second thm}}

\begin{proof}[Proof of Theorem \ref{main thm}]
Throughout we assume $X\ll (\log t)^A$. By Stirling's formula we have $\Gamma(\tfrac{1}{2}-\sigma+iy)\ll y^{-\sigma} e^{-\frac{\pi}{2}|y|}$ and 
$\frac{\Gamma^\prime}{\Gamma}(\tfrac{1}{2}-\sigma+iy)\ll \log(2+|y|)$. Applying these along with the classical bound $S(t)\ll \log t$ we find that 
\begin{multline*}
          X^{1/2-\sigma}\int_{C\log\log t}^{t/2} |\Gamma(\tfrac{1}{2}-\sigma+iy)|\Big(\log X+\Big|\frac{\Gamma^\prime}{\Gamma}(\tfrac{1}{2}-\sigma+iy) \Big|\Big)|S(t+y)|dy  \\
    \ll   X^{1/2-\sigma}(\log t)(\log X) \int_{C\log\log t}^{t/2} e^{-\frac{\pi}{2}y}\log (2+y)dy \\
    \ll   X^{1/2-\sigma}(\log t)^{1-\frac{\pi}{2}C}\log_2 t\log_3 t.
\end{multline*}
Hence, for $C>1/\pi$  we may restrict the range of integration in $J(t, X)$ to $|y|\leqs C\log\log t$ at the cost of an error of size $O(X^{1/2-\sigma}\sqrt{\log t})$. Then, using similar bounds to estimate this remaining integral gives
\[
J(t,X)\ll X^{1/2-\sigma}(\log X) L(t)+X^{1/2-\sigma}\sqrt{\log t}.
\]

Applying this in \eqref{explicit form} and integrating from $\sigma=1$ to $\sigma=9/8$ we get
\begin{multline}
\label{log zeta}
      \log\zeta(1+it)= 
      \sum_{n\geqs 1}\frac{\Lambda(n)}{n^{1+it}\log n}e^{-n/X}+ O\bigg(\bigg|\log\zeta(9/8+it)-\sum_{n\geqs 1}\frac{\Lambda(n)}{n^{9/8+it}\log n}e^{-n/X}\bigg|\bigg)
 \\  +O\bigg(\frac{L(t)}{X^{1/2}}\bigg)
+ O\Big(\frac{\sqrt{\log t}}{X^{1/2}\log X}\Big)  + O\Big(\frac{\log t}{X}\Big).
\end{multline}
Choosing $X$ such that $\max(L(t)^2,\log t)=o(X)$ we see that the error terms in the second line of the above are all $o(1)$. It remains to consider the sum over primes. 

By splitting the sum at $n=X$ and applying the expansion $e^{-n/X}=1+O(n/X)$ in the sum over $n\leqs X$ we find that, for $\sigma\geqs 1$,
\be
\begin{split}
        \sum_{n\geqs 1}\frac{\Lambda(n)}{n^{\sigma+it}\log n}e^{-n/X}
 % = & \sum_{n\leqs X}\frac{\Lambda(n)}{n^{\sigma+it}\log n}\big(1+O(n/X)\big) +O(X^{-\sigma})  \\
 % = & \sum_{n\leqs X}\frac{\Lambda(n)}{n^{\sigma+it}\log n}+O\bigg(\frac{1}{X}\sum_{p^k\leqs X}1\bigg)+O(X^{-\sigma}) \\
    = & \sum_{n\leqs X}\frac{\Lambda(n)}{n^{\sigma+it}\log n}+O\bigg(\frac{1}{\log X}\bigg)
\end{split}
\ee
after estimating the tail sum by the prime number theorem. Also note that for $\sigma\geqs 1$
\be
\begin{split}
          \sum_{n\leqs X}\frac{\Lambda(n)}{n^{\sigma+it}\log n}
    = & \sum_{\substack{k\geqs 1,p\leqs X\\p^k\leqs X}}\frac{1}{kp^{k(\sigma+it)}}
    =  -\sum_{p\leqs X}\log\bigg(1-\frac{1}{p^{\sigma+it}}\bigg)+O\bigg(\sum_{\substack{k\geqs 1,p\leqs X\\p^k> X}}\frac{1}{kp^{k\sigma}}\bigg)
  \\ = & -\sum_{p\leqs X}\log\bigg(1-\frac{1}{p^{\sigma+it}}\bigg)+O\bigg(\frac{1}{\log X}%\sum_{\substack{p\leqs X}}\frac{\log p}{p(p-1)}
  \bigg).
\end{split}
\ee
From this it is clear that the first error term of \eqref{log zeta} is $o(1)$ as $X\to\infty$ and hence we acquire the asymptotic
\[
    \zeta(1+it)=(1+o(1))\prod_{p\leqs X}\bigg(1-\frac{1}{p^{1+it}}\bigg)^{-1} 
\]
provided $\max(L(t)^2,\log t)=o(X)$. Theorem \ref{main thm} then follows.

\end{proof}

\begin{proof}[Proof of Theorem \ref{second thm}] Adapting the above proof to the case $1/2<\sigma<1$ gives
\be
\label{log zeta sigma}
      \log\zeta(\sigma+it)= 
      \sum_{n\geqs 1}\frac{\Lambda(n)}{n^{\sigma+it}\log n}e^{-n/X} +O\bigg(\frac{L(t)}{X^{\sigma-1/2}}\bigg)
  + O\Big(\frac{\sqrt{\log t}}{X^{\sigma-1/2}\log X}\Big)+ O\Big(\frac{\log t}{X}\Big)
  %+ O\Big(\frac{1}{\log X}\Big).
\ee
A short calculation with the prime number theorem shows that the sum over $n$ is $\ll X^{1-\sigma}/\log x$ and so 
\ben
\label{log zeta sigma}
      \log|\zeta(\sigma+it)|\ll
      \frac{X^{1-\sigma}}{\log X} +\frac{L(t)}{X^{\sigma-1/2}}
  +\frac{\sqrt{\log t}}{X^{\sigma-1/2}\log X}+\frac{\log t}{X}.
\een
Taking $X=\max(L(t)^2,\log t)(\log\log t)^2$ balances the first two terms and the result follows.
\end{proof}

%%%%%%%% SECTION PROOF OF PROP %%%%%%%%%%

\section{Proof of Proposition \ref{main prop}}

We start from a slightly more precise version of the explicit formula used by Littlewood (see Theorem 14.4 of Titchmarsh \cite{T}). The proof is fairly standard but we shall give most of the details for clarity.

\begin{lem}
Assume the Riemann hypothesis. For large $t$ we have
\ben
\label{explicit}
-\frac{\zeta^\prime(s)}{\zeta(s)}=\sum_{n\geqs 1}\frac{\Lambda(n)}{n^s}e^{-n/X}+\sum_{\rho}X^{\rho-s}\Gamma(\rho-s)+O(X^{-1}\log t)
\een
uniformly for $1/2\leqs \sigma \leqs 9/8$ and $1\leqs X\leqs e^{\sqrt{t}}$.
\end{lem} 
\begin{proof}
On the one hand we have
\[
\sum_{n\geqs 1}\frac{\Lambda(n)}{n^s}e^{-n/X}=-\frac{1}{2\pi i}\int_{2-i\infty}^{2+i\infty}\Gamma(z-s)\frac{\zeta^\prime(z)}{\zeta(z)}X^{z-s}dz
\]
which follows from the identity $e^{-Y}=\frac{1}{2\pi i}\int_{(c)}\Gamma(z)Y^{-z}dz$. On the other hand, by shifting contours to the left in the usual way we find that
\begin{multline*}
    -\frac{1}{2\pi i}\int_{2-i\infty}^{2+i\infty}\Gamma(z-s)\frac{\zeta^\prime(z)}{\zeta(z)}X^{z-s}dz\\
= -\frac{\zeta^\prime(s)}{\zeta(s)}-\sum_{\rho}X^{\rho-s}\Gamma(\rho-s)+\Gamma(1-s)X^{1-s}\\
   + \frac{\zeta^\prime(s-1)}{\zeta(s-1)}X^{-1} +\frac{1}{2\pi i}\int_{(\kappa)}\Gamma(z-s)\frac{\zeta^\prime(z)}{\zeta(z)}X^{z-s}dz
\end{multline*}
where $-2+\sigma <\kappa<-1+\sigma$. 

Now,
\[
\Gamma(1-s)X^{1-s}\ll e^{-A|t|}X^{1/2}\leqs e^{-B|t|}
\]
and 
\[
\frac{\zeta^\prime(s-1)}{\zeta(s-1)}X^{-1}\ll \frac{\log t}{X}
\]
since $\zeta^\prime(s)/\zeta(s)\ll \log t$ for $-1\leqs \sigma\leqs 2$ provided that $\sigma\neq 1/2$. 
The integral on the new line is 
\be
\begin{split}
\ll  X^{\kappa-\sigma}\int_{-\infty}^\infty|\Gamma(\kappa+iy-s)|\bigg|\frac{\zeta^\prime(\kappa+iy)}{\zeta(\kappa+iy)}\bigg|dy 
\ll & {X^{\kappa-\sigma}}\int_{-\infty}^\infty e^{-A|y-t|}\log(|y|+2)dy \\
\ll & {X^{\kappa-\sigma}} \log t.
\end{split}
\ee
Clearly this is smaller than our previous error term so we are done.
\end{proof}

One may conduct some basic estimates of the sum over zeros appearing in \eqref{explicit} which gives an upper bound of $X^{1/2-\sigma}\log t$ (see section 14.5 of \cite{T}). After integrating over $\sigma\geqs 1$, we see that one requires $X=\log^2 t$ for this term to be $o(1)$. This is the reason for such a restriction in the length of the Euler product in \eqref{euler prod}. As mentioned, we would like to exploit some cancellation in the sum over zeros in a hope to improve this. In this direction we have the following Lemma.

\begin{lem}Assume the Riemann hypothesis and let $t$ be large. Then, uniformly for $1/2+\delta\leqs \sigma \leqs 9/8$ with fixed $\delta>0$, and $1\leqs X\leqs e^{\sqrt{t}}$, 
we have
\[
 \sum_{\rho}X^{\rho-s}\Gamma(\rho-s)=-\frac{\log (t/2\pi)}{X}+J(t,X)+O(X^{-2}\log t)+ O(X^{1/2-\sigma})
\]
where 
\[
J(t,X)=iX^{1/2-\sigma}\int_{-t/2}^{t/2} X^{iy}\Gamma(\tfrac{1}{2}-\sigma+iy)\Big(\log X+\frac{\Gamma^\prime}{\Gamma}(\tfrac{1}{2}-\sigma+iy) \Big)S(t+y)dy.
\]
\end{lem}

\begin{proof}
We first note that we may restrict the sum to those ordinates for which $t/2\leqs\gamma\leqs 3t/2$. For, the tail satisfies the bound
\[
\sum_{\gamma<t/2}X^{\rho-s}\Gamma(\rho-s)\ll X^{1/2-\sigma} \sum_{\gamma<t/2}e^{-A|\gamma-t|}\ll X^{1/2-\sigma}e^{-At/4}\sum_{\gamma<t/2}e^{-A|\gamma|/2}\ll X^{1/2-\sigma}e^{-At/4}
\]
after writing this last sum as a Stieltjes integral and applying the appropriate bounds on $N(t)$. A similar bound holds for the sum over $\gamma>3t/2$. 

We write the remaining sum in the form
\be
\begin{split}
         \sum_{t/2\leqs \gamma\leqs3t/2}X^{\rho-s}\Gamma(\rho-s)
  = & X^{1/2-\sigma}\sum_{t/2\leqs \gamma\leqs3t/2}X^{i(\gamma-t)}\Gamma(\tfrac{1}{2}-\sigma+i(\gamma-t)) \\
  = & X^{1/2-\sigma}\int_{-t/2}^{3t/2} X^{i(y-t)}\Gamma(\tfrac{1}{2}-\sigma+i(y-t))dN(y).
\end{split}
\ee
We decompose $N(y)$ as a sum of its smooth part and $S(y)$; that is, we write 
$
N(y)=N^*(y)+S(y)
$
where
\be
\begin{split}
N^*(y)&=\frac{1}{\pi}\Delta\arg s(s-1)\pi^{-s/2}\Gamma(s/2)\\
          &=\frac{y}{2\pi}\log(y/2\pi e)+\frac{7}{8}+\frac{c}{y}+O(1/y^2),\qquad y\to\infty
\end{split}
\ee
and $S(y)=\frac{1}{\pi}\Delta\arg\zeta(s)$. Here, $\Delta\arg$ denotes the change in argument along the straight lines from 2 to $2+iy$, and then to $1/2+iy$. 
We note that $N^*(y)$ is a smooth function and its above asymptotic expansion can be given to any degree of accuracy in terms of negative powers of $y$. 

Then, our integral can be written as
\be
\begin{split}
    &  X^{1/2-\sigma}\int_{t/2}^{3t/2} X^{i(y-t)}\Gamma(\tfrac{1}{2}-\sigma+i(y-t))[dN^*(y)+dS(y)] \\
 = & X^{1/2-\sigma}\int_{t/2}^{3t/2} X^{i(y-t)}\Gamma(\tfrac{1}{2}-\sigma+i(y-t))\Big(\frac{1}{2\pi}\log (y/2\pi)+O(1/y^2)\Big)dy \\
     & +X^{1/2-\sigma}\int_{t/2}^{3t/2} \frac{d}{dy}\Big(X^{i(y-t)}\Gamma(\tfrac{1}{2}-\sigma+i(y-t))\Big)S(y)dy\\
     &\qquad+O(X^{1/2-\sigma}e^{-Bt})+O(X^{1/2-\sigma}/t) \\
\end{split}
\ee
after integration by parts and applying the bounds $\Gamma(\tfrac{1}{2}-\sigma+iy)\ll e^{-A|y|}$ and $S(y)\ll \log y$. Denote the first of these integrals by $I$ and the second by $J$.
Now, 
\be
\begin{split}
I = & X^{1/2-\sigma}\int_{t/2}^{3t/2} X^{i(y-t)}\Gamma(\tfrac{1}{2}-\sigma+i(y-t))\Big(\frac{1}{2\pi}\log (y/2\pi)+O(1/y^2)\Big)dy \\ 
      = & X^{1/2-\sigma}\int_{-t/2}^{t/2} X^{iy}\Gamma(\tfrac{1}{2}-\sigma+iy)\frac{1}{2\pi}\log ((y+t)/2\pi)dy+O(X^{1/2-\sigma}/t) \\
      = & \log\Big(\frac{t}{2\pi}\Big)\frac{1}{2\pi}\int_{-t/2}^{t/2} X^{1/2-\sigma+iy}\Gamma(\tfrac{1}{2}-\sigma+iy)dy\\
         & +  X^{1/2-\sigma}\int_{-t/2}^{t/2} X^{iy}\Gamma(\tfrac{1}{2}-\sigma+iy)\frac{1}{2\pi}\log ((1+y/t)/2\pi)dy+ O(X^{1/2-\sigma}/t) 
\end{split}
\ee
The second integral here is bounded and so results in a contribution of $O(X^{1/2-\sigma})$. After extending the tails of the first integral, which incurs only a small error, we acquire 
\[
I=\log\Big(\frac{t}{2\pi}\Big)\frac{1}{2\pi i}\int_{1/2-\sigma-i\infty}^{1/2-\sigma+i\infty}\Gamma(s)X^sds+O(X^{1/2-\sigma}).
\]
In the usual way we may shift this contour to the far left encountering poles at $s=-n$, $n\in\mathbb{N}$ with residues $(-1)^n/n!$. Since $\sigma>1/2$ there is no contribution from the pole at zero. In this way we find that this integral is given by $e^{-1/X}-1$ and so 
\be
\begin{split}
I = & (e^{-1/X}-1)\log(t/2\pi)+ O(X^{1/2-\sigma}) \\
      = & -\frac{\log (t/2\pi)}{X}+O(X^{-2}\log t)+ O(X^{1/2-\sigma}).
\end{split}
\ee

Performing the differentiation in $J$ gives
\[
J=iX^{1/2-\sigma}\int_{t/2}^{3t/2} X^{i(y-t)}\Gamma(\tfrac{1}{2}-\sigma+i(y-t))\Big(\log X+\frac{\Gamma^\prime}{\Gamma}(\tfrac{1}{2}-\sigma+i(y-t)) \Big)S(y)dy.
\]
and then on substituting $y\mapsto y+t$ the result follows. 
\end{proof}

Combining the above two lemmas gives Proposition \ref{main prop}. Note that the integral $I$ above is trivially $\ll X^{1/2-\sigma}\log t$. Evaluating it explicitly is where we acquire some cancellations, however the problem is then reduced to finding good bounds on $S(t)$, of which we know very little.

\end{document}